\documentclass[12pt]{amsart}
\usepackage{amsmath,amsthm,amsfonts,amssymb}

\newcommand{\bbZ}{\mathbb{Z}}
\newcommand{\bbR}{\mathbb{R}}
\newcommand{\bbQ}{\mathbb{Q}}
\newcommand{\bbP}{\mathbb{P}}
\newcommand{\bbF}{\mathbb{F}}
\newcommand{\bbA}{\mathbb{A}}
\newcommand{\bbN}{\mathbb{N}}

\newcommand{\calO}{\mathcal{O}}

\newcommand{\calH}{\mathcal{H}}
\newcommand{\calG}{\mathcal{G}}
\newcommand{\calP}{\mathcal{P}}

\newcommand{\gl}{\mathfrak{gl}}

\newcommand{\red}{\operatorname{red}}

\newcommand{\Ndim}{\operatorname{Ndim}}
\newcommand{\Spec}{\operatorname{Spec}}

\newcommand{\GL}{\mathrm{GL}}

\newcommand{\SL}{\mathrm{SL}}

\newtheorem{thm}{Theorem}

\newtheorem{prop}[thm]{Proposition}

\newtheorem{defn}[thm]{Definition}

\newtheorem{lem}[thm]{Lemma}

\newtheorem*{claim*}{Claim}
\begin{document}

\title{P-adic Nori Theory}
\author{Michael Larsen}
\address{Michael Larsen,
Department of Mathematics,
Indiana University,
Bloomington, IN
U.S.A. 47401}
\thanks{The author was partially supported by NSF grants DMS-0354772
and DMS-0800705.}

\begin{abstract}
Given a fixed integer $n$, we consider closed subgroups $\calG$ of $\GL_n(\bbZ_p)$, where
$p$ is sufficiently large in terms of $n$.
Assuming that the Zariski closure of $\calG$ in $\GL_n$ has no toric part, we give a condition on 
the (mod $p$) reduction of $\calG$ which guarantees that $\calG$ is of bounded index in
$\GL_n(\bbZ_p)\cap G(\bbQ_p)$.
\end{abstract}

\maketitle

In \cite{No}, Nori considered a special class of subgroups of $\GL_n(\bbF_p)$, namely groups which are generated by elements of order $p$ or, as we shall say, \emph{$p$-generated groups}.  He showed that if
$p$ is sufficiently large in terms of $n$, there is a correspondence between $p$-generated groups and a
certain class of connected algebraic groups which he called \emph{exponentially generated}.  In particular, every $p$-generated group $\Gamma$ is a subgroup of $G(\bbF_p)$ for the corresponding algebraic group $G$, and $[G(\bbF_p):\Gamma]$ is bounded by a constant depending only on $n$.  The $p$-generated groups are admittedly rather special, but on the other hand, every finite subgroup
$\Gamma \subset \GL_n(\bbF_p)$ contains a $p$-generated normal subgroup, $\Gamma^+$,
of prime-to-$p$ index, which shows that every $\Gamma$ can be related to a connected algebraic group in a weak sense.  This construction can serve in some measure as a substitute for the (identity component of the) 
Zariski-closure in the setting of finite linear groups, where the actual identity component of the Zariski-closure of $\Gamma$
is always trivial.  

In this paper we consider closed subgroups $\calG$ 
of the compact $p$-adic Lie group $\GL_n(\bbZ_p)$.
In this setting, of course, Zariski-closure behaves well, so we do not need a substitute.  Nevertheless, it turns out that there is an interesting class of groups $\calG$ for which we can prove a bounded index result analogous to that of Nori.  

Throughout the paper, $n$ will denote a positive integer and $F$ a field.   If $F$ is of characteristic $p>0$, we assume $p\ge n$, so $i!$ is non-zero for $i < n$.
As every nilpotent element $x\in M_n(F)$ satisfies
$x^n = 0$, the truncated exponential function
$$\exp(x) := \sum_{i=0}^{n-1} \frac{x^i}{i!}$$
satisfies $\exp(x+y) = \exp(x)\exp(y)$ for every pair $x,y$ of commuting nilpotent matrices.
Moreover $\exp(x)-1$ is nilpotent, so $\exp(x)$ is unipotent.  Conversely, if $u$ is unipotent, $1-u$ is nilpotent, so 
$$\log(u) := -\sum_{i=1}^{n-1} \frac{(1-u)^i}{i}$$
is nilpotent, and $\log$ and $\exp$ set up mutually inverse bijections between the unipotent and
nilpotent $n\times n$ matrices over $F$.  In the positive characteristic case, every unipotent element 
$u\neq 1$ is of order $p$, and conversely, every element of order $p$ is unipotent (because this is true for every Jordan block of order $\le p$).

For every nilpotent element $x\in M_n(F)$, there exists a morphism of algebraic groups
$\phi_x\colon \bbA^1\to \GL_n$
defined by
$$\phi_x(t) := \exp(t x).$$
If $x\neq 0$, this morphism is injective, and its image is isomorphic to $\bbA^1$.
If $N$ is a set of nilpotent elements of $M_n(F)$, let $G_N$ denote the subgroup of
$\GL_n$ generated by $\phi_x(\bbA^1)$ for all $x\in N$, i.e., the intersection of all
algebraic subgroups of $\GL_n$ which contain
$$\bigcup_{x\in N} \phi_x(\bbA^1).$$
Following Nori we say that an algebraic subgroup of
$\GL_n$ over a field $F$ is \emph{exponentially generated} if it is of the form $G_N$ for some
$N\subset M_n(F)$.

\begin{prop} 
\label{eg-implies-sbi}
Every exponentially generated group is the extension of
a semisimple group by a unipotent group.
\end{prop}

\begin{proof}
It is clear that every quotient group of an exponentially generated group must be
generated by subgroups isomorphic to the additive group. In
particular exponentially generated groups must be connected, and every reductive exponentially generated group must be semisimple since no nontrivial torus is generated by
additive groups. It follows that the quotient of any exponentially generated group by
its unipotent radical is  semisimple.
\end{proof}

In general, the converse of Proposition~\ref{eg-implies-sbi} is not true.  For example, if $F=\bbR$,
$\GL_n$ contains compact semisimple subgroups which have no non-trivial unipotent elements.
If $F$ is of positive characteristic, even if it is algebraically closed,
the image of $\SL_2$ under the $4$-dimensional representation which is the direct sum of the standard representation and its Frobenius twist
fails to be exponentially generated.
In characteristic zero, we have a precise criterion for exponential generation.

\begin{prop}
\label{sbi-implies-eg}
Let $F$ be a field of characteristic zero.
An algebraic subgroup $G$ of $\GL_n$ defined over $F$
is exponential generated if and only if it has 
has no non-trivial finite, toric, or anisotropic quotient group.
\end{prop}

\begin{proof}
As there is no non-trivial homomorphism from an additive group to a finite, toric, 
or anisotropic group, one direction is clear.  For the other,
let $U$ denote a unipotent $F$-subgroup of $G$.  Thus $U$ has a composition
series 
$$U = U_0\supset U_1\supset\cdots\supset U_s = \{e\}$$ 
with each $U_i/U_{i+1}$
isomorphic to the additive group.  By Steinberg's theorem \cite{St},
$H^1(F,U_i) = 0$ for all $i$, so for each 
$1\le i\le s$ we have a short exact sequence
$$0\to U_i(F)\to U_{i-1}(F)\to F\to 0,$$
and there exists $u_{i-1}\in U_{i-1}(F)\setminus U_i(F)$.
As $F$ is of characteristic zero, 
$$\langle u_i\rangle \subset U_{i-1}(F)\cap \phi_{\log(u_i)}(F)$$
is isomorphic to $\bbZ$, so $\phi_{\log(u_i)}(\bbA^1)\cap U_{i-1}$ has dimension $1$, which means
$\phi_{\log(u_i)}(\bbA^1)\subset U_{i-1}$.
It follows that
$$\phi_{\log(u_i)}(\bbA^1)U_i = U_{i-1}.$$
Thus, by descending induction,
$$U = \prod_{i=1}^s\phi_{\log(u_{i-1})}(\bbA^1).$$

Let $H$ denote the quotient of $G=G^\circ$ by its unipotent radical $N$.
As $H$ is isotropic, the set $\calP$ of its
proper parabolic $F$-subgroups is non-empty.
For each $P\in\calP$, let $U_P$ denote the inverse image in $G$ of the
unipotent radical of $P$.
Thus each $U_P$ is a unipotent $F$-subgroup of $G$ containing $N$.  
Each is therefore exponentially generated.  Let $K\subset G$ be the (exponentially generated group)
generated by all $U_P$.  Thus $K$ is normalized by the inverse image of $H(F)$
in $G$.
By a theorem of Chevalley \cite{Ch}, $H(F)$ is Zariski-dense in $H$, so $K$ is normal in
$G$.  Thus $G/K$ is isomorphic to a quotient $H/(K/N)$, which is isotropic.
It follows that $\calP$ contains a proper parabolic $F$-subgroup not contained in $K/N$, contrary
to assumption.  Thus $K=G$, and $G$ is exponentially generated.

\end{proof}

We say that a Lie algebra is \emph{nilpotently generated} if it is spanned by its nilpotent elements.
Nori proved \cite[Theorem~A]{No} that if $F$ is of characteristic zero or characteristic 
$p$ sufficiently large in terms of $n$, the $\log$ and $\exp$ maps give
mutually inverse bijections, described more explicitly below, 
between exponentially generated $F$-subgroups of $\GL_n$
and nilpotently generated $F$-subalgebras of the Lie algebra $M_n = \gl_n$.

The following proposition allows us to put all 
exponentially generated subgroups (as well, possibly,  as other
subvarieties of $\GL_n$) into a family over a base of finite type.  It is convenient to 
work projectively, by embedding $\GL_n$ into $\bbP^{n^2}$.
For any scheme $Z$ and any closed subvariety $K$ of $\GL_{n,Z}$, we denote by $\bar K$
the closed subset $Z\cup (\bbP^{n^2}_Z\setminus \GL_{n,Z})$ endowed with its reduced induced scheme structure.

\begin{prop}\label{exp-flat} For every positive integer $n$ 
there exists an integer $N$ and a finite set $S$ of polynomials
such that for every field $F$
over $\bbZ[1/N]$ and every exponentially generated subgroup
$G_F\subset \GL_{n, F}$, the Hilbert polynomial of $\bar G_F$ belongs to $S$.
\end{prop}

\begin{proof}
We prove that there exists a positive integer $N$ and a morphism $Y'\to X'$
of schemes of finite type over $\bbZ$ such that for all $F$ whose characteristic does
not divide $N$ and all exponentially generated $G_F\subset \GL_{n,F}$,
there exists $x'\in X'(F)$ with $Y'_{x'} = \bar G_F$.
By \cite[\S2]{SB221}, the set of Hilbert polynomials for the $\bar G_F$ is therefore finite.

We begin by trying to parametrize nilpotently generated Lie algebras.
The set of $k$-tuples of nilpotent $n\times n$ matrices which span a Lie
subalgebra of $n\times n$ matrices is constructible because Lie algebra closure can be expressed as the existence
of a set of $k^3$ structure constants for the Lie bracket.  
Let $N_n/\bbZ$ denote the scheme of nilpotent $n\times n$ matrices and
$W\subset N^{n^2}_n$ the constructible set of ordered $n^2$-tuples of
nilpotent matrices spanning a Lie algebra.  
Replacing $W$ with the disjoint union $X$ of 
the strata of a suitable stratification, we get a scheme 
indexing $n^2$-tuples of nilpotent matrices which span nilpotent Lie algebras.
Thus, for every field $F$ of characteristic zero or characteristic $p$ sufficiently large
and every nilpotently generated Lie algebra $L\subset \gl_n$ over $F$, there exists
$x\in X(F)$ which indexes a spanning set of $L$.

We choose $N$ sufficiently divisible that
outside of characteristics dividing $N$, there is
a bijection between exponentially generated
subgroups $G$ of $\GL_n$ and nilpotently generated Lie subalgebras $L$ of
$\gl_n$, given by the mutually inverse maps
sending $G$ to its Lie algebra and $L$ to the group generated by $\phi_x(\bbA^1)$
for all nilpotent $x\in L$.  In particular, $\phi_{x_i}(\bbA^1)$ generates $G$ whenever
$x_1,\ldots,x_{n^2}$ is a nilpotent spanning set of $L$.  
From the scheme $X$ indexing all possible $n^2$-tuples,
we would like to obtain a scheme of finite type over $\bbZ[1/N]$ indexing all $\bar G_F$, where $G_F$ ranges over exponentially generated groups and $F$ ranges over fields over
$\bbZ[1/N]$.

Recall \cite[Proposition~2.2]{Bo} that if $V\subset G\subset \GL_n$ is any connected generating subvariety of an algebraic group $G$, 
the image of $V^{n^2}$ under the multiplication map is dense in $G$,
and the image of $V^{2n^2}$ is exactly $G$ .
This implies
$$(\phi_{x_1}(\bbA^1)\cdots\phi_{x_{n^2}}(\bbA^1))^{2n^2} \twoheadrightarrow G.$$

Let $Y:= \bbP^{n^2}_X$ and
$$Z := (\bbP^{n^2}_{X}\setminus \GL_{n,X})
\coprod (X\times \bbA^{2n^4}).$$
We define $\xi\colon Z \to Y$
by extending the obvious inclusion map on the first component of $Z$ by
$$\xi((x_1,\ldots,x_{n^2}),(t_{1,1},\ldots,t_{n^2,2n^2})) := 
((x_1,\ldots,x_{n^2}),\prod_{j=1}^{2n^2} \prod_{i=1}^{n^2} \phi_{x_i}(t_{i,j})).$$
For each $F$ and each $x\in X(F)$, the image of the map of fibers 
$Z_x\to Y_x = \bbP^{n^2}_F$ is the union of $\bbP^{n^2}_F\setminus \GL_{n,F}$ 
and the exponential subgroup of $\GL_{n,F}$
in correspondence with the nilpotently generated Lie subalgebra 
of $\gl_n(F)$ associated to $x$.  
The following lemma now implies the proposition.

\end{proof}

\begin{lem}
Let  $m$ be a positive integer, $X$ a scheme of finite type over $\bbZ$, $Y$ a closed subscheme of
$\bbP^m_X$, and $\xi\colon Z\to Y$ a morphism of finite type such that $\xi(Z_x)$ is a closed subset of $Y_x$ for all $x\in X$.
There exists $N\in\bbN$, a morphism $\psi\colon X'\to X$, and a closed subscheme
$Y'\subset \bbP^m_{X'}$ such that for every field $F$ over $\bbZ[1/N]$ and every $x\in X(F)$,
there exists $x'\in X'_x(F)$ such that $Y'_{x'} = \xi(Z_x)^{\red}$.
\end{lem}

\begin{proof}
We use Noetherian induction on $X$.
If the image of $Z\to X$ has Zariski-closure $C\subsetneq X$, we can replace $X$ and $Y$
by $C$ and $Y_C$ respectively.  We therefore assume without loss of generality that
$Z\to X$ has dense image.  Replacing $Z$ by $Z^{\red}$, without loss of generality we
may assume $Z$ is reduced.
We choose $N$ divisible by every prime which is the characteristic of a generic point of $X$.

Let $\eta$ denote a generic point of $X$.  As any localization of a reduced ring is reduced,
$Z_\eta$ is reduced.  Either $\eta$ lies over a prime $p$ dividing $N$
or $\eta$ is of characteristic zero.  
In the former case, let $U_1$ denote any neighborhood of $\eta$ which lies over 
$\Spec\bbF_p$.
In the latter case, $Z_\eta$ is geometrically reduced
\cite[Proposition~4.6.1]{EGA},  so
$Z_x$ is geometrically reduced for all $x$ in some neighborhood $U_1$ of $\eta$
\cite[Theorem~9.7.7~(iii)]{EGA}.  
Let $W$ denote the Zariski-closure of $\overline{\xi(Z)}\setminus \xi(Z)$ in $Y$, endowed
with its reduced induced scheme structure.
As $\xi(Z_\eta)$ is closed in $Y_\eta$, the $\eta$-fibers of $\xi(Z)$ and $\overline{\xi(Z)}$
are the same, so $W_\eta$ is empty.  Let $U_2$ denote a neighborhood of $\eta$ which
does not meet the image of $W\to X$.  Finally, let $U = U_1\cap U_2$,
$X_1 = X\setminus U$, $Y_1 = Y\times_X X_1$, $Z_1 = Z\times_X X_1$.

By the induction hypothesis, if $N$ is sufficiently divisible, 
the lemma holds for $X_1$, $Y_1$, and $Z_1$.
Let $X'_1$, $Y'_1$, and $\psi_1$ be chosen suitably.  Let $X' = U\coprod X'_1$ and $Y' = W_U\coprod Y'_1$, and let $\psi$ denote the extension of $\psi_1$ which is given on $W_U$ by the composition
of the obvious maps $W_U\to Y\to \bbP^m_X\to X$.  If $x\in X(F)$ belongs to $X_1(F)$, we are done already.  If not, it belongs to $U(F)$.  Let $x'$ denote the image of $x\in U(F)$ under the inclusion
$U\to X'$. 
As $U\subset U_2$, at the set level, the fiber $Y'_{x'}$ coincides with $\xi(Z_x)$.
As $U\subset U_1$, if $F$ is a $\bbZ[1/N]$-algebra, then $Y'_{x'}$ is reduced.

\end{proof}

We now specialize to the case $F = \bbF_p$, where $p\ge n$.
If $\Gamma$ is a subgroup of $\GL_n(\bbF_p)$, we write $\Gamma^+$ for the
subgroup of $\Gamma$ generated by all elements of order $p$.
Let $N(\Gamma) = N(\Gamma^+)$ denote the set $\{\log u\mid u^p=1, u\in \Gamma\}$, and
let $G := G_{N(\Gamma)}$.  Then $\Gamma^+ \subset G(\bbF_p)$.

\begin{defn}
If $\Gamma$ is a subgroup of $\GL_n(\bbF_p)$ we define the 
\emph{Nori dimension}, $\Ndim(\Gamma)$, to be $\dim G_{N(\Gamma)}$.
 Likewise if $\calG$ is a subgroup of
$\GL_n(\bbZ_p)$ its \emph{Nori dimension}, $\Ndim(\calG)$, is the  Nori dimension of
its reduction$\pmod p$.
\end{defn}

\begin{lem} Let $p\ge 2n$, $x$ a nilpotent $n\times n$ matrix over $\bbF_p$, and
$A \in \GL_n(\bbZ_p)$ a $p$-adic lift of $\exp(x)$. Then for all positive integers $k$,
\[A^{p^k}\equiv 1 + p^k M\pmod{p^{k+1}}\] where $M$ reduces$\pmod p$
to $x$.
\end{lem}

\begin{proof} It suffices to prove the lemma when $k=1$.  
Without loss of generality, we may assume that $M$ is nilpotent, so $M^p = 0$.
Let $N = \exp(M)-1$.  As $N$ reduces (mod $p$) to the nilpotent element
$\exp(x)-1$, $N^n$ is divisible by
$p$ in $M_n(\bbZ_p)$, and we can write $A$ as $1+N+pB $ for some $B\in M_n(\bbZ_p)$.
Expanding,
\begin{equation*}
\begin{split}
A^p & =(1+N+pB)^p = \sum_{m=0}^p \binom pm (N+pB)^m \\
&\equiv \sum_{m=0}^p \binom pm\Bigl[N^m + p\sum_{i+j=m-1} N^i B N^j\Bigr] \\
&\equiv \sum_{m=0}^p \binom pm N^m = (1+N)^p = \exp(pM) \equiv 1 + pM \pmod{p^2}.
\end{split}
\end{equation*}

\end{proof}

\begin{thm}\label{NDZD}
For every positive integer $n$ there exist constants $A_n$, $B_n$, and $C_n$ such that if $p>A_n$
is prime, $\calG$ is a closed subgroup of $\GL_n(\bbZ_p)$, and $G$
is the Zariski closure of $\calG$ in $  \GL_n$, then $\Ndim(\calG)\le \dim G$.  
If $\Ndim(\calG) = \dim G$, then: 
\begin{enumerate}
\item $\calG$ an open subgroup of $G(\bbQ_p)$.
\item $G/G^\circ$ is of prime-to-$p$ order and has a normal abelian subgroup of index $\le B_n$.
\item If, in addition, the radical of $G^\circ$ is unipotent, then 
$$[G(\bbQ_p)\cap \GL_n(\bbZ_p):\calG] \le C_n.$$
\end{enumerate}
\end{thm}

\begin{proof}
We fix $A_n\ge 2n$ large enough for Proposition~\ref{exp-flat} to apply.

Let $\calH =G(\bbQ_p) \cap \GL_n(\bbZ_p)$. Let $F_m \calH $ denote
the subgroup of $ \calH $ consisting of elements congruent to
$1$\relax$\pmod{p^m}$. We identify $F_m \calH /F_{m+1}\calH$ with
a subspace of $M_n $ over the field $\bbF_p$.  As
$$ \left(1+p^m A\right)^p\equiv 1 +p^{m+1} A\pmod{p^{m+2}},$$
we have that 
$$F_m \calH/ F_{m+1} \calH \subset F_{m+1}\calH /F_{m+2}\calH$$
for all $m\ge 1$.  It follows that 
$$\dim F_m \calH/ F_{m+1}\calH \leq \dim G$$
for all $m\ge 1$.  Indeed, otherwise, the quotient
$\calH/F_m \calH $ would grow at least as fast as $c p^{m(1+\dim G)}$, 
which is impossible \cite[Thm.~8]{Se}.

As $\calG\subset\calH$, we have
$$F_m\calG /F_{m+1}\calG \subset F_m\calH/F_{m +1} \calH.$$
By the
preceding lemma the dimension of $F_m\calG / F_{m+1}\calG $ is at
least the dimension of the vector space spanned by the logarithms of
elements of order $p$ in the$\pmod{p}$ reduction of $\calG$. By the
correspondence between exponentially generated groups and
nilpotently generated Lie algebras this dimension is the Nori
dimension of $\calG$.  In summary, for all $m\ge 1$,
$$\Ndim(\calG)\le F_m\calG/F_{m+1}\calG\le F_m\calH/F_{m+1}\calH\le \dim G.$$
This proves the first claim of the theorem.

If the Nori dimension of $\calG$ equals $\dim G$, we have further that
$$\dim F_m \calG / F_{m+1} \calG = \dim F_m\calH / F_{m+1}  \calH,$$
for all $m\ge 1$. 
As $\calG$ and $\calH$ are closed subgroups of $\GL_n(\bbZ_p)$, this implies
$F_1\calG = F_1\calH$, which implies (1).

If $G$ is any closed subgroup of $\GL_n$, there exists a finite central extension
of $G/G^\circ$ which can be realized as a subgroup of $G(\bbQ_p)$.
(See, e.g., the proof of \cite[Proposition~6.2]{KLS}.)
Jordan's theorem implies the existence of a normal abelian subgroup of bounded index.  

For $n<p-1$, $\GL_n(\bbQ_p)$
has no element of order $p$, since the $p$th cyclotomic polynomial is
irreducible over $\bbQ_p$.  On the other hand, every extension of a group containing
an element of order $p$ again has an element of order $p$.  This gives (2).

For (3), we note first that since $\calG$ meets every component of $G$,
it suffices to prove that 
$$\calG^\circ:=\calG\cap G^\circ(\bbQ_p)$$
is of bounded index in $G^\circ(\bbQ_p)\cap\GL_n(\bbZ_p)$.
As $[\calG:\calG^\circ]$ is prime to $p$, the$\pmod p$ reduction of
$\calG^\circ$ is of prime-to-$p$ index in that of $\calG$.  It follows
that $\Ndim(\calG^\circ) = \Ndim(\calG)$.  Replacing $\calG$ with
$\calG^\circ$ if necessary, we may assume without loss of generality that
$G$ is connected.  

Let $F$ denote any finite extension of $\bbQ_p$ over which $G$ has no non-trivial 
anisotropic quotient.
We may take $F$ to be totally ramified over $\bbQ_p$ since the anisotropic 
simple groups over $\bbQ_p$ are all central quotients of groups of the form
 $\SL_1(D)$, where $D$ is a division algebra
over $\bbQ_p$ \cite{Kn}, and every degree $n$ division 
algebra over $\bbQ_p$ splits over $\bbQ_p(p^{1/n})$.
We denote by $\calO$ the ring of elements of non-negative valuation in $F$.
Thus, the residue field of $\calO$ is $\bbF_p$.
By Proposition~\ref{sbi-implies-eg}, 
$G_F$ is exponentially generated.

Let $\bar G_F$ denote $G_F\cup (\bbP^{n^2}_F\setminus \GL_{n,F})$,
regarded as a reduced subscheme of $\bbP^{n^2}_F$
and $\bar G_{\calO}$ denote the schematic closure of
$\bar G_F\subset \bbP^{n^2}_F$ in 
$\bbP^{n^2}_{\calO}$, i.e., the unique $\calO$-flat closed subscheme of
$\bbP^{n^2}_{\calO}$ having generic fiber $\bar G_F$ \cite[Proposition~2.8.5]{EGA}.
Thus, $\calH\subset \bar G_{\calO}(\calO)$.  

Let $X$ denote the union of Hilbert schemes of the polynomials in $S$ over $\bbZ[1/N]$, where $N$ and $S$ are given by Proposition~\ref{exp-flat}.
Let $Y$ be the universal closed subscheme of $\bbP^{n^2}_X$ with Hilbert polynomials in $S$.
If $A_n$ is sufficiently large, for every $p>A_n$, every $p$-adic field $F$, and every exponentially generated $G_F\subset \GL_{n,F}$, there exists an $F$-point $x\in X(F)$ such that
$G_F = Y_x\cap \GL_{n,F}$.  By the valuative criterion of properness, $x$ extends to a 
morphism $\Spec\calO\to X$, where $\calO$ is the ring of integers in $F$.  Pulling back $Y$ by this morphism, we obtain an $\calO$-flat subscheme of $\GL_{n,\calO}$ whose generic point is  $\bar G_F$.  
This must be isomorphic to $\bar G_{\calO}$ by uniqueness of flat extension over $\calO$.  
Let $G_{\calO}$ denote the intersection of $\bar G_{\calO}$ with 
$\GL_{n,\calO}\subset \bbP^{n^2}_{\calO}$.  Thus $G_{\calO}$ is flat over $\calO$ and
the generic fiber of $G_{\calO}$ is $\bar G_F\cap \GL_{n,F}=G_F$.
The fiber $G_{\bbF_p}$ has no more irreducible components
than the fiber $\bar G_{\bbF_p}$, which can be regarded as a fiber of $Y\to X$.
By the local constructibility of the function giving the number of irreducible components of 
geometric fibers \cite[Corollary~9.7.9]{EGA} and Noetherian induction, 
this gives an upper bound $d_n$ on $G_{\bbF_p}/G_{\bbF_p}^\circ$ independent of
$G$ and $p>A_n$.

By the flatness of $G_{\calO}$,
the special fiber $G_{\bbF_p}$ has dimension equal to that of $G_F$, which is $\Ndim(\calG)$.  
We claim that the number of $\bbF_p$-points of a connected $d$-dimensional 
algebraic group over $\bbF_p$ is at least $(p-1)^d$ and at most $(p+1)^d$.
This is obvious for additive groups (where the number of points is $p^d$) and tori
(where the number of points is $Q(p)$, $Q$ the characteristic polynomial
of Frobenius on the character group), and it is well-known in the semisimple case.
It follows in the general case from the structure theory of connected linear algebraic groups.
The upper bound implies
$$G_{\bbF_p}(\bbF_p) \le |G_{\bbF_p}/G^\circ_{\bbF_p}| (p+1)^{\Ndim(\calG)}
\le d_n (3/2)^{n^2} p^{\Ndim(\calG)}.$$
The kernel $F_1G_{\calO}(\calO)$ of the reduction map 
$$G_{\calO}(\calO)\to G_{\calO}(\bbF_p) = G_{\bbF_p}(\bbF_p)$$
consists of elements of $F_1\GL_n(\calO)$, i.e., elements of $\GL_n(\calO)$ 
congruent to $1$ modulo the maximal ideal of $\calO$.
Thus, 
$$\calH\cap F_1G_{\calO}(\calO)\subset \GL_n(\bbZ_p)\cap F_1\GL_n(\calO)
= F_1\GL_n(\bbZ_p).$$
It follows that
$$|\calH/F_1\calH| \le d_n (3/2)^{n^2}p^{\Ndim(\calG)}.$$
On the other hand, by Nori's theorem \cite{No}, $(\calG/F_1\calG)^+$ is of bounded index $e_n$
in $G_{N(\calG/F_1\calG)}(\bbF_p)$.  The lower bound for points on a connected group 
implies
$$|\calG/F_1\calG|\ge |(\calG/F_1\calG)^+| \ge e_n^{-1}(p-1)^{\Ndim(\calG)} \ge e_n^{-1}2^{-n^2} p^{\Ndim(\calG)}.$$

Combining these estimates, we obtain
$$\frac{|\calH/F_1\calH|}{|\calG/F_1\calG|} \le 3^{n^2}d_ne_n.$$
As $F_1\calG = F_1\calH$, setting $C_n = 3^{n^2}d_ne_n$, we obtain (3).
\end{proof}

\end{document}